\documentclass[12pt]{amsart}
\usepackage{amsmath,amsthm,amsfonts,amssymb,latexsym,enumerate,url,hyperref} 

\begin{document}

\theoremstyle{plain}

\newtheorem{thm}{Theorem}[section]
\newtheorem{lem}[thm]{Lemma}
\newtheorem{conj}[thm]{Conjecture}
\newtheorem{pro}[thm]{Proposition}
\newtheorem{cor}[thm]{Corollary}
\newtheorem{que}[thm]{Question}
\newtheorem{rem}[thm]{Remark}
\newtheorem{defi}[thm]{Definition}
\newtheorem{hyp}[thm]{Hypothesis}

\newtheorem*{thmA}{THEOREM A}
\newtheorem*{corB}{COROLLARY B}
\newtheorem*{corC}{COROLLARY C}

\newtheorem*{thmC}{THEOREM C}
\newtheorem*{conjA}{CONJECTURE A}
\newtheorem*{conjB}{CONJECTURE B}
\newtheorem*{conjC}{CONJECTURE C}

\newtheorem*{thmAcl}{Main Theorem$^{*}$}
\newtheorem*{thmBcl}{Theorem B$^{*}$}

\numberwithin{equation}{section}

\newcommand{\Maxn}{\operatorname{Max_{\textbf{N}}}}
\newcommand{\Syl}{\operatorname{Syl}}
\newcommand{\dl}{\operatorname{\mathfrak{d}}}
\newcommand{\Con}{\operatorname{Con}}
\newcommand{\cl}{\operatorname{cl}}
\newcommand{\Stab}{\operatorname{Stab}}
\newcommand{\Aut}{\operatorname{Aut}}
\newcommand{\Ker}{\operatorname{Ker}}
\newcommand{\IBr}{\operatorname{IBr}}
\newcommand{\Irr}{\operatorname{Irr}}
\newcommand{\SL}{\operatorname{SL}}
\newcommand{\FF}{\mathbb{F}}
\newcommand{\NN}{\mathbb{N}}
\newcommand{\N}{\mathbf{N}}
\newcommand{\C}{\mathbf{C}}
\newcommand{\OO}{\mathbf{O}}
\newcommand{\F}{\mathbf{F}}

\renewcommand{\labelenumi}{\upshape (\roman{enumi})}

\newcommand{\GL}{\operatorname{GL}}
\newcommand{\Sp}{\operatorname{Sp}}
\newcommand{\PGL}{\operatorname{PGL}}
\newcommand{\PSL}{\operatorname{PSL}}
\newcommand{\SU}{\operatorname{SU}}
\newcommand{\PSU}{\operatorname{PSU}}
\newcommand{\PSp}{\operatorname{PSp}}

\providecommand{\V}{\mathrm{V}}
\providecommand{\E}{\mathrm{E}}
\providecommand{\ir}{\mathrm{Irr_{rv}}}
\providecommand{\Irrr}{\mathrm{Irr_{rv}}}
\providecommand{\re}{\mathrm{Re}}

\def\irrp#1{{\rm Irr}_{p'}(#1)}
\def\ibrp#1{{\rm IBr}_{p'}(#1)}

\def\Z{{\mathbb Z}}
\def\C{{\mathbb C}}
\def\Q{{\mathbb Q}}
\def\irr#1{{\rm Irr}(#1)}
\def\ext#1{{\rm Ext}(#1)}
\def\ibr#1{{\rm IBr}(#1)}
\def\irra#1{{\rm Irr}_{\rm A}(#1)}
\def\ibra#1{{\rm IBr}_{\rm A}(#1)}
\def \c#1{{\cal #1}}
\def\cent#1#2{{\bf C}_{#1}(#2)}
\def\syl#1#2{{\rm Syl}_#1(#2)}
\def\nor{\trianglelefteq\,}
\def\oh#1#2{{\bf O}_{#1}(#2)}
\def\Oh#1#2{{\bf O}^{#1}(#2)}
\def\zent#1{{\bf Z}(#1)}
\def\det#1{{\rm det}(#1)}
\def\ker#1{{\rm ker}(#1)}
\def\norm#1#2{{\bf N}_{#1}(#2)}
\def\alt#1{{\rm Alt}(#1)}
\def\iitem#1{\goodbreak\par\noindent{\bf #1}}
   \def \mod#1{\, {\rm mod} \, #1 \, }
\def\sbs{\subseteq}

\def\gc{{\bf GC}}
\def\ngc{{non-{\bf GC}}}
\def\ngcs{{non-{\bf GC}$^*$}}
\newcommand{\notd}{{\!\not{|}}}
\newcommand{\Out}{{\mathrm {Out}}}
\newcommand{\Mult}{{\mathrm {Mult}}}
\newcommand{\Inn}{{\mathrm {Inn}}}
\newcommand{\IBR}{{\mathrm {IBr}}}
\newcommand{\IBRL}{{\mathrm {IBr}}_{\ell}}
\newcommand{\IBRP}{{\mathrm {IBr}}_{p}}
\newcommand{\ord}{{\mathrm {ord}}}
\def\id{\mathop{\mathrm{ id}}\nolimits}
\renewcommand{\Im}{{\mathrm {Im}}}
\newcommand{\Ind}{{\mathrm {Ind}}}
\newcommand{\diag}{{\mathrm {diag}}}
\newcommand{\soc}{{\mathrm {soc}}}
\newcommand{\End}{{\mathrm {End}}}
\newcommand{\sol}{{\mathrm {sol}}}
\newcommand{\Hom}{{\mathrm {Hom}}}
\newcommand{\Mor}{{\mathrm {Mor}}}
\newcommand{\St}{{\sf {St}}}
\def\rank{\mathop{\mathrm{ rank}}\nolimits}
\newcommand{\Tr}{{\mathrm {Tr}}}
\newcommand{\tr}{{\mathrm {tr}}}
\newcommand{\Gal}{{\it Gal}}
\newcommand{\Spec}{{\mathrm {Spec}}}
\newcommand{\ad}{{\mathrm {ad}}}
\newcommand{\Sym}{{\mathrm {Sym}}}
\newcommand{\Char}{{\mathrm {char}}}
\newcommand{\pr}{{\mathrm {pr}}}
\newcommand{\rad}{{\mathrm {rad}}}
\newcommand{\abel}{{\mathrm {abel}}}
\newcommand{\codim}{{\mathrm {codim}}}
\newcommand{\ind}{{\mathrm {ind}}}
\newcommand{\Res}{{\mathrm {Res}}}
\newcommand{\Ann}{{\mathrm {Ann}}}
\newcommand{\Ext}{{\mathrm {Ext}}}
\newcommand{\Alt}{{\mathrm {Alt}}}
\newcommand{\AAA}{{\sf A}}
\newcommand{\SSS}{{\sf S}}
\newcommand{\CC}{{\mathbb C}}
\newcommand{\CB}{{\mathbf C}}
\newcommand{\RR}{{\mathbb R}}
\newcommand{\QQ}{{\mathbb Q}}
\newcommand{\ZZ}{{\mathbb Z}}
\newcommand{\KK}{{\mathbb K}}
\newcommand{\NB}{{\mathbf N}}
\newcommand{\ZB}{{\mathbf Z}}
\newcommand{\OB}{{\mathbf O}}
\newcommand{\EE}{{\mathbb E}}
\newcommand{\PP}{{\mathbb P}}
\newcommand{\GC}{{\mathcal G}}
\newcommand{\HC}{{\mathcal H}}
\newcommand{\AC}{{\mathcal A}}
\newcommand{\BC}{{\mathcal B}}
\newcommand{\GA}{{\mathfrak G}}
\newcommand{\SC}{{\mathcal S}}
\newcommand{\TC}{{\mathcal T}}
\newcommand{\DC}{{\mathcal D}}
\newcommand{\LC}{{\mathcal L}}
\newcommand{\RC}{{\mathcal R}}
\newcommand{\CL}{{\mathcal C}}
\newcommand{\EC}{{\mathcal E}}
\newcommand{\GCD}{\GC^{*}}
\newcommand{\TCD}{\TC^{*}}
\newcommand{\FD}{F^{*}}
\newcommand{\GD}{G^{*}}
\newcommand{\HD}{H^{*}}
\newcommand{\hG}{\hat{G}}
\newcommand{\hP}{\hat{P}}
\newcommand{\hQ}{\hat{Q}}
\newcommand{\hR}{\hat{R}}
\newcommand{\GCF}{\GC^{F}}
\newcommand{\TCF}{\TC^{F}}
\newcommand{\PCF}{\PC^{F}}
\newcommand{\GCDF}{(\GC^{*})^{F^{*}}}
\newcommand{\RGTT}{R^{\GC}_{\TC}(\theta)}
\newcommand{\RGTA}{R^{\GC}_{\TC}(1)}
\newcommand{\Om}{\Omega}
\newcommand{\eps}{\epsilon}
\newcommand{\varep}{\varepsilon}
\newcommand{\al}{\alpha}
\newcommand{\chis}{\chi_{s}}
\newcommand{\sigmad}{\sigma^{*}}
\newcommand{\PA}{\boldsymbol{\alpha}}
\newcommand{\gam}{\gamma}
\newcommand{\lam}{\lambda}
\newcommand{\la}{\langle}
\newcommand{\ra}{\rangle}
\newcommand{\hs}{\hat{s}}
\newcommand{\htt}{\hat{t}}
\newcommand{\sgn}{\mathsf{sgn}}
\newcommand{\SR}{^*R}
\newcommand{\tn}{\hspace{0.5mm}^{t}\hspace*{-0.2mm}}
\newcommand{\ta}{\hspace{0.5mm}^{2}\hspace*{-0.2mm}}
\newcommand{\tb}{\hspace{0.5mm}^{3}\hspace*{-0.2mm}}
\def\skipa{\vspace{-1.5mm} & \vspace{-1.5mm} & \vspace{-1.5mm}\\}
\newcommand{\tw}[1]{{}^#1\!}
\renewcommand{\mod}{\bmod \,}

\newcommand{\Irre}[1]{{\rm Irr}(#1)}
\newcommand{\Irra}[2]{{\rm Irr}_{#1}(#2)}
\renewcommand{\N}[2]{{{\bf N}_{#1}(#2)}}
\def\st{\text{ }|\text{ }}
\newcommand{\inv}[1]{{#1}^{-1}}

\marginparsep-0.5cm

\renewcommand{\thefootnote}{\fnsymbol{footnote}}
\footnotesep6.5pt

\title{The Fields of Values of the Isaacs' Head Characters}
\author{Gabriel Navarro}
\address{Departament de Matem\`atiques, Universitat de Val\`encia, 46100 Burjassot,
Val\`encia, Spain}
\email{gabriel@uv.es}

\thanks{This research is supported by Grant PID2022-137612NB-I00
 funded by MCIN/AEI/ 10.13039/501100011033 and ERDF ``A way of making Europe".
 I thank N. Rizo and L. Sanus for a careful reading of this manuscript.}

\dedicatory{To I. M. Isaacs, in memoriam}



\begin{abstract}
We determine the fields of values of the Isaacs' head characters of a finite solvable group.
 \end{abstract}

\maketitle

\section{Introduction} In one of his final major papers (\cite{Is22}), I. M. Isaacs introduced a canonical subset ${\rm H}(G)
\sbs \irr G$ of the irreducible characters of a finite solvable group $G$, in bijection with the linear characters
${\rm Lin}(C)$ of a Carter subgroup $C$ of $G$.  (Discovered by R. W. Carter in \cite{C61}, recall that $C$ is any self-normalizing nilpotent subgroup of $G$ and that any two of them are $G$-conjugate \cite{C61}.)
Isaacs named them the {\sl head characters} of $G$. As a corollary, he proved that $|C/C'| \le k(G)$,
where $C'=[C,C]$ is the derived subgroup of $C$ and $k(G)$ is the number of conjugacy classes of $G$. A direct proof of this result is not immediately apparent.

Subsequently, I constructed a canonical subset ${\rm Irr}_{\mathfrak F'}(G)$ of the complex characters of a finite solvable group $G$, associated with every formation $\mathfrak F$, and in bijection with $\irr{\norm GH/H'}$, where $H$ is an $\mathfrak F$-projector of $G$  (\cite{N22}). When $\mathfrak F$ is the class of $p$-groups,
 then ${\rm Irr}_{\mathfrak F'}(G)$ is the set of irreducible characters of $G$ of degree not divisible by $p$
 and $H$ is a Sylow $p$-subgroup of $G$; when $\mathfrak F$ is the class of nilpotent groups, then 
 $H=C$ is a Carter subgroup of $G$ and ${\rm Irr}_{\mathfrak F'}(G) = {\rm H}(G)$. Thus, Isaacs' result can be seen as the {\sl McKay conjecture} for the formation of nilpotent groups.

There are many open questions about this intriguing set ${\rm H}(G)$. For instance, whether these characters can be detected from the character table remains an open question. It is also unknown whether ${\rm H}(G)$ consists of the irreducible characters of $G$ that do not have any zero value on the Carter subgroup $C$. Isaacs did prove that if $\eta \in {\rm H}(G)$, then $\eta_K \in \irr K$ whenever $K \nor G$ and $G/K$ is nilpotent, and that $\eta(1)$ divides $|G:C|$, but these properties certainly do not
characterize the set ${\rm H}(G)$. On a slightly different subject, it is not even known if the character table 
of a solvable group $G$ determines the size $|C|$ of the Carter subgroup $C$, the number
of linear caracters $|C/C'|$ of $C$, or even the primes dividing $|C|$.
Hence, a determination of ${\rm H}(G)$ in the character table sounds difficult.

In this note, we offer a new property of the Isaacs head characters and present two consequences
beyond their world. Recall that if $\chi$ is a character, then the field of values of $\chi$, $\Q(\chi)$,
 is the smallest subfield of the complex numbers containing the values of $\chi$. Although $|{\rm H}(G)| = |C/C'|$, no canonical bijection ${\rm Lin}(C) \rightarrow {\rm H}(G)$ is known
 and therefore, a priori, no connection between the
 field of values of these characters can be made. Nevertheless, we establish the following global/local result.

\begin{thmA}
Let $G$ be a finite solvable group, and let $C$ be a Carter subgroup of $G$. Then the fields of  values of the head characters are exactly the
$n$--th cyclotomic fields $\Q_n$ corresponding to the orders $n$ of the linear characters of $C$.
\end{thmA}

We see Theorem A as the {\sl Galois-McKay theorem} for the formation of nilpotent groups.
We remind the reader that the analog of Theorem A
for the formation of $p$-groups (that is, the McKay original case)   does not in general hold,  as shown,
 for instance, by $G={\rm GL}_2(3)$ and $p=3$. (Here $\norm GP/P'$ is a rational group, and $G$ has irreducible characters of degree 2
and field of values $\Q(i\sqrt 2)$.)   It is interesting to notice that the Galois--McKay conjecture (\cite{N04})
requires to consider field of values of characters over the $p$-adics, while the analog for
formations containing the nilpotent groups seems to
work well over the rationals. (The corresponding {\sl head characters} associated to those, together with some other properties of the head characters have been recently considered in \cite{FGS25}.)
\medskip

As a consequence of Theorem A, we prove the following.

\begin{corB}
Suppose that $G$ is a solvable group, and let $C$ be a Carter subgroup of $G$.
Then $G$ has at least $m$ rational-valued irreducible characters, where $m=|P/\Phi(P)|$,
$P\in \syl 2 C$ and $\Phi(P)$ is the Frattini subgroup of $P$.
\end{corB}

We can also use Theorem A to reprove a fact on rational groups,   which is not too
well-known (and which does not seem to have a direct group theoretical
proof).
Recall that a rational group is a group whose irreducible characters are all rational-valued.

\begin{corC}
Suppose that $G$ is a solvable rational group. Then the Carter subgroups of $G$ are the Sylow 2-subgroups of $G$.
\end{corC}
 
At the end of this note, we discuss to what extent these corollaries can be extended to general finite groups.

\section{The key result}
For characters, we use the notation from \cite{Is} and \cite{N18}.
Recall that a {\sl Carter subgroup} $C$ of a solvable group
is a self-normalizing nilpotent subgroup of $G$. Also, any two of them are $G$-conjugate,
and if $N\nor G$, then $CN/N$ is a Carter subgroup of $G/N$ (\cite{C61}).

\medskip
We begin with a known result, which may not be familiar to every character theorist. It concerns formations (see \cite{DH92} IV.5.18) and, if conveniently applied, could have simplified some of the results in \cite{Is22}.
We present the following elementary proof in the case of interest, including it here for the reader's convenience. 
 
\begin{lem}\label{dh}
Let $G$ be a finite solvable group. Let $K$ be the smallest normal subgroup
of $G$ such that $G/K$ is nilpotent.
If $K$ is abelian, then $K$ is complemented in $G$, 
and its complements are the Carter subgroups of $G$.
\end{lem}

\begin{proof}
Since $G/K$ is nilpotent, then $G=KC$.
We  show that $K\cap C=1$.  We argue by induction on $|G|$.
Suppose that $|K|$ is not a a prime power. Let $p$ be any prime dividing $|K|$.
Let $1<K_{p'}$ be the Hall $p$-complement of $K$.
By induction, we have that $CK_{p'} \cap K=K_{p'}$.
Therefore, if $K_p$ is a Sylow $p$-subgroup of $K$, we have that
 $C \cap K_p \sbs CK_{p'} \cap K \cap K_p=1$.
Now if $1 \ne c \in C \cap K$ and $p$ is any prime dividing $o(c)$,
we have that  $c_p \in C \cap K_p$ (because $K$ has a normal Sylow
$p$-subgroup $K_p$), and this is a contradiction.
Therefore, we may assume that $K$ is a $p$-group for some prime $p$.
Since $G/K$ is nilpotent, then $G$ has a normal Sylow $p$-subgroup.
Write $C=C_p \times H$, where $C_p \in \syl pC$.  Notice that $H$ is a $p$-complement of $G$
and that $KC_p$ is the normal Sylow $p$-subgroup of $G$. Notice that $KH \nor G$, since
$G/K$ is nilpotent. Hence $[KC_p,H] \sbs K$, and $[KC_p,H]=[KC_p, H,H]=[K,H]$, by coprime action.
Thus $D=[K,H] \nor (KC_p)H=G$. Thus  $G/D$ has a nilpotent normal $p$-complement,
and a normal Sylow $p$-subgroup, and we conclude that $G/D$ is nilpotent.
Hence $D=K$ by hypothesis. 
Since $K=[K,H] \times \cent KH$, by coprime action
(and using that $K$ is abelian), we deduce that $\cent KH=1$.  
Now,  since $KH \nor G$, we have that $G=K\norm GH$ by the Frattini argument. Since $\cent KH=1$,
it follows that $\norm GH$ complements $K$ in $G$. Notice then that $\norm GH$
is nilpotent, because it is isomorphic to $G/K$. Also notice that $\norm GH$ is
self-normalizing, by the Frattini argument (since $H$ is a Hall $p$-complement of $G$).
Therefore $C$ and $\norm GH$ are $G$-conjugate. Since $C \sbs \norm GH$, we conclude that
$C=\norm GH$.
\end{proof}
We shall use several times the following result.
If a group $A$ acts on automorphisms on another group $G$, we
denote by ${\rm Irr}_A(G)$ the set of $A$-invariant irreducible
characters of $G$.

\begin{thm}\label{is}
Suppose that $G$ is a finite solvable group and let $C$ be a Carter subgroup
of $G$. Let $K$ be the smallest normal subgroup of $G$ such that $G/K$ is nilpotent, and let $L\nor G$
such that $K/L$ is abelian.
\begin{enumerate}[(a)]
\item
If $\theta \in {\rm Irr}_C(K)$, then the restriction $\theta_L$ contains a unique
$C$-invariant constituent $\theta^\prime \in {\rm Irr}_C(L)$.

\item
If $\varphi \in {\rm Irr}_C(L)$, then the induced character $\varphi^K$ contains
a unique $C$-invariant irreducible constituent $\tilde \varphi \in {\rm Irr}_C(K)$.

\item
The maps $^\prime: {\rm Irr}_C(K) \rightarrow
{\rm Irr}_C(L)$ and $\, \tilde{}: {\rm Irr}_C(L) \rightarrow {\rm Irr}_C(K)$ are inverse bijections.
\item
We have that
$\theta \in {\rm Irr}_C(K)$ extends to $G$ if and only if $\theta^\prime$ extends to $LC$.
\end{enumerate}
\end{thm}

\begin{proof}
By Lemma \ref{dh}, we have that $K\cap LC=L$.
Then this follows from  Theorem 3.1 and Lemma 2.1 of \cite{Is22}.
\end{proof}

If $N \nor G$ and $\theta \in \irr N$, we denote by ${\rm Ext}(G|\theta)$ the (possibly empty)
set of $\chi \in \irr G$ such that $\chi_N=\theta$. The following is the key result in this paper.

\begin{thm}\label{main}
Suppose that $G$ is a finite solvable group, and let $C$ be a Carter subgroup of $G$. Let $K$ be the smallest normal subgroup of $G$
such that $G/K$ is nilpotent.  Let $L\nor G$ be such that $K/L$ is abelian.
Let $\theta \in \irr K$ be $G$-invariant, and let $H=CL$.
Let $\varphi=\theta^\prime \in \irr L$ be the unique $C$-invariant irreducible constituent of $\theta_L$
(by Theorem \ref{is}).
Then there exists a bijection $^{*}: {\rm Ext}(G|\theta) \rightarrow {\rm Ext}(H|\varphi)$
such that $\Q(\chi)=\Q(\chi^*)$.
\end{thm}

\begin{proof}
We argue by induction by $|G:L|$.
By Lemma \ref{dh}, we have that $H\cap K=L$. Also, we have that
$G=KC$. By Theorem \ref{is}, there is a unique irreducible $C$-invariant constituent $\varphi$ of the restriction $\theta_L$,
$\theta$ is the unique $C$-invariant irreducible constituent of $\varphi^K$,  and
${\rm Ext}(G|\theta)$ is not empty if and only if ${\rm Ext}(H|\varphi)$ is not empty.
Hence, we may assume that $\theta$ extends to $G$ and that $\varphi$ extends to $H$.
We claim that $\Q(\theta)=\Q(\varphi)$. Indeed, let $\sigma \in {\rm Gal}(\Q_{|G|}/\Q)$.
By elementary Galois theory, it suffices to show that $\theta^\sigma=\theta$ if 
and only if $\varphi^\sigma=\varphi$. But this is clear because Galois action commutes with
conjugation and by the uniqueness of $\varphi$ and $\theta$ with respect to each other.

If $L<E<K$ is a normal subgroup of $G$, then $\theta_E$ has a unique $C$-invariant irreducible
 constituent. Also $\eta_L$ has a unique irreducible $C$-invariant constituent which necessarily is $\varphi$. By induction, we easily may assume that $K/L$ is a chief factor of $G$.
 Hence, $K/L$ is an abelian $p$-group and $H$ is a maximal subgroup of $G$.
 
 We have that $L \sbs H \sbs G_\varphi$. Therefore $G_\varphi=H$ or $G_\varphi=G$.
 Assume first that $G_\varphi=H$. Therefore $K_\varphi=L$ and thus $\varphi^K=\theta$.
 By the Clifford correspondence (Theorem 6.11 of \cite{Is}), we have that
 induction defines a bijection $\irr{H|\varphi} \rightarrow \irr{G|\varphi}=\irr{G|\theta}$.
 Hence, induction defines a bijection
 $\ext{H|\varphi} \rightarrow \ext{G|\theta}$.
 We claim that $\Q(\psi)=\Q(\psi^G)$ for $\psi \in \irr{H|\varphi}$.
 By the induction formula, we have that $\Q(\psi^G) \sbs \Q(\psi)$.
 Let $\sigma \in {\rm Gal}(\Q(\psi)/\Q(\psi^G))$.
 Since $\sigma$ fixes $\psi^G$, it follows that $\sigma$ fixes $\theta$.
 Hence $\sigma$ fixes $\varphi$, since $\Q(\theta)=\Q(\varphi)$. Now $\psi^\sigma \in \irr{G_\varphi|\varphi}$
 induces $\psi^G$. By the uniqueness in the Clifford correspondence, we have that
 $\psi^\sigma=\psi$. Hence $\Q(\psi)=\Q(\psi^G)$.
 
 We may assume that $G_\varphi=G$. By the Going-Down theorem (6.18 of \cite{Is}),
 we have that either $\theta_L=\varphi$, or $\theta_L=e\varphi$, with $e^2=|K:L|$.
  Suppose first that $\theta_L=\varphi$. By Lemma 6.8(d) of \cite{N18},
 we have that restriction defines a bijection $\irr{G|\theta} \rightarrow \irr{H|\varphi}$.
 We claim that $\Q(\chi)=\Q(\chi_H)$ for $\chi \in \irr{G|\theta}$. 
 We clearly have that $\Q(\chi_H) \sbs \Q(\chi)$. Suppose now that
 $\sigma \in {\rm Gal}(\Q(\chi)/\Q(\chi_H))$.  Since $\sigma$ fixes $\chi_H$ it follows that
 $\sigma$ fixes $\varphi$. Then $\sigma$ fixes $\theta$, and we have that $\chi$ and $\chi^\sigma$ lie
 over $\theta$ and restrict irreducibly to the same character of $H$. Hence, $\chi^\sigma=\chi$,
 by uniqueness, and the claim is proven.
 
 We are left with the case where $\theta_L=e\varphi$, where $e^2=|K:L|$.
 Suppose that $p$ divides $|G:K|$, and let 
  $P/K$ be   a $p$-group chief factor of $G$ (using that $G/K$ is nilpotent).
 Let $Q=P\cap H$. Since $P/L$ is a $p$-group, we have that $(K/L) \cap \zent{P/L}>1$, 
 and since this is normal in $G$, we conclude that $K/L\sbs \zent{P/L}$.
 Therefore,  we have that $Q \nor G$. Notice that $P/Q$ is a chief factor and that $G/Q$ is not nilpotent
 (since otherwise $G/(K\cap Q)=G/L$ would be nilpotent).

 Let $\Delta=\{\chi_P \, |\, \chi \in \ext{G|\theta}\}$, and let $\Xi=\{\tau_Q \, |\, \tau \in \ext{H|\varphi}\}$.
 We have that
 $$\ext{G|\theta}=\bigcup_{\eta \in \Delta} \ext{G|\eta}$$
 and 
 $$\ext{H|\varphi}=\bigcup_{\gamma \in \Xi} \ext{H|\gamma}$$
 are disjoint unions. 
 It is enough to find a bijection $\eta \mapsto \eta^\prime$ 
 from $\Delta \rightarrow \Xi$ 
and another bijection $^*:\ext{G|\eta} \rightarrow \ext{H|\eta^\prime}$ that preserves fields of values.
Let $\eta \in \Delta$. Since $\eta$ is $G$-invariant,  let $\eta^\prime$ be the unique $C$-invariant irreducible constituent of $\eta_Q$
(by Theorem \ref{is}). Notice that $\eta^\prime$ lies over $\varphi$. Again, by this theorem, we have that $\eta^\prime$
extends to $H$.  We know that $\varphi$ extends to $H$ and therefore
to $Q$. Since $Q/L$ is abelian, it follows that all the irreducible constituents
of $\eta^Q$ are extensions of $\varphi$ (using Gallagher's theorem, Corollary 6.17 of \cite{Is}).
Therefore, $\eta^\prime$ extends $\varphi$,
and therefore $\eta^\prime \in \Xi$. Suppose now that $\gamma \in \Xi$,
and let $\tilde\gamma$ be the unique irreducible constituent of $\gamma^P$ which is $C$-invariant.
Notice that $\tilde\gamma$ lies over $\theta$ (because it lies over $\varphi$, and $\varphi^K=e\theta$). We know that $\theta$ extends to $G$, so it extends to $P$. By the same argument before,
we have that all irreducible constituents of $\theta^P$ are extensions of $\theta$, and therefore,
we conclude that $\tilde\gamma$ extends $\theta$. Hence $\tilde\gamma \in \Delta$,
and our map $$\Delta \rightarrow \Xi$$ given by $\eta \mapsto \eta^\prime$ is a bijection.
 By induction, we are done in this case. 
 
 So we may assume that $G/K$ is a $p'$-group.
 In this case, the theorem follows from Theorems 9.1 and 6.3 of \cite{Is73}.
 Indeed, Theorem 9.1 of \cite{Is73} guarantees the existence of a character $\Psi$ of $H/L$
 satisfying that $|\Psi(h)|^2=|\cent{K/L} h|$ and such that the equation
 $$\chi_H=\Psi \chi^*$$
 defines a bijection $^*:\irr{G|\theta} \rightarrow \irr{H|\varphi}$.
 In fact, this gives a bijection  $^*:\ext{G|\theta} \rightarrow \ext{H|\varphi}$
 by computing degrees.
 In our coprime situation, the fact that the values of $\Psi$ 
 are rational is given after Corollary 6.4 of \cite{Is73}. 
 Also, Theorem 9.1(d) of \cite{Is73} asserts that $\chi(g)=0$ if $g$ is not
 $G$-conjugate to an element of $H$, and we easily deduce that $\Q(\chi)=\Q(\chi^*)$.
 This completes the proof of the theorem.
\end{proof}

\section{Head Characters}
We briefly remind the reader on how head characters are constructed following the
approach in \cite{N22} when $\mathfrak F$ is the formation of nilpotent groups.
Suppose that $G$ is a finite solvable group and fix a Carter subgroup $C$ of $G$. We are going
to define a canonical subset ${\rm H}(G) \sbs \irr G$ of every subgroup of
$G$ that contains $C$ as a Carter subgroup,  by induction on $|G|$.
If $G$ is nilpotent, that is, if $C=G$, then we let ${\rm H}(G)={\rm Lin}(G)$.
Suppose now that $G$ is not nilpotent. Let  $K$ be the smallest normal subgroup of
$G$ such that $G/K$ is nilpotent and $L=K'$   the derived subgroup of $K$.
Again notice that $G=KC$ and $K \cap LC=L$ by Lemma \ref{dh}.
Thus $G=LC$ if and only if $K=L$, which happens if and only if $G$ is nilpotent.
Hence, $H=LC<G$.  We have then defined ${\rm H}(H)$ by induction.
Let $^{\prime}: {\rm Irr}_C(K) \rightarrow {\rm Irr}_C(L)$ be the canonical bijection defined in Theorem \ref{is}.
By Theorem 4.3(b) of \cite{N22}, we have that
${\rm H}(G)$ is the set of $\tau \in \irr G$ such that $\tau_K=\theta \in \irr K$, and 
$\theta^\prime \in \irr L$ lies below any character $\psi \in {\rm H}(H)$. Using this result and induction,
we can easily prove that if $\lambda \in \irr G$ is linear, then $\lambda \in {\rm H}(G)$ and $\lambda {\rm H}(G)={\rm H}(G)$.
(To prove all these results one can also use Corollary 3.7 and Lemma 5.6 of \cite{Is22}.)
Finally, by Theorem 4.3(c) of \cite{N22} (or Corollary 5.7 of \cite{Is22}), if $\tau \in {\rm H}(G)$ and $G/M$ is nilpotent, 
where $M \nor G$, we have that $\tau_M \in \irr M$.

\medskip

Now we are ready to prove Theorem A.

\begin{thm}
Let $G$ be a solvable group, and let $C$ be a Carter subgroup of $G$.
Then there is a bijection $^*:{\rm Lin}(C) \rightarrow {\rm H}(G)$ satisfying
$\Q(\chi)=\Q(\chi^*)$.
\end{thm}
\begin{proof}
We argue by induction on $|G|$. If $G$ is nilpotent, then $C=G$
and ${\rm H}(G)={\rm Lin}(G)$, and there is nothing to prove.
Suppose that $G$ is not nilpotent. Let $K$ be the smallest normal subgroup
of $G$ such that $G/K$ is nilpotent, let $L=K'$ and let $H=LC < G$.
By induction, there is a bijection
$^*:{\rm Lin}(C) \rightarrow {\rm H}(H)$ satisfying
$\Q(\chi)=\Q(\chi^*)$. Hence, it suffices to show that there is
a bijection $^*:{\rm H}(H) \rightarrow {\rm H}(G)$ that respects fields of values.
Let $\Xi=\{ \tau_L \, |\, \tau \in {\rm H}(H)\}$. By Theorem 4.3. of \cite{N22}, we have that 
$\tau_L \in \irr L$. Let $\Delta=\{\theta \in \irr K |\, \theta^\prime \in \Xi\}$.
As we have said before, we have that
$${\rm H}(G)=\bigcup_{\theta \in \Delta} \ext{G|\theta} \, .$$
Also, 
$${\rm H}(H)=\bigcup_{\eta \in \Xi} \ext{H|\eta} \, ,$$
since ${\rm H}(H)$ is closed under multiplication by linear characters.
Now the result easily follows by Theorem \ref{main}.
\end{proof}

The following is Corollary B.
\begin{cor}\label{corC}
Suppose that $G$ is a solvable group, and let $C$ be a Carter subgroup of $G$.
Then $G$ has at least $m$ rational-valued irreducible characters, where $m=|P/\Phi(P)|$,
and $P\in \syl 2 C$.
\end{cor}

\begin{proof}
It easily follows from Theorem A, since all characters of $P/\Phi(P)$ are
rational--valued, whenever $P$ is a $2$-group.
\end{proof}
\medskip

The following is a slight generalization of Corollary C.

\begin{cor}\label{corB}
Suppose that $G$ is a  solvable group whose only cyclotomic field of values of characters
is the rationals. Then the
Sylow 2-subgroups are self-normalizing. In particular, they are the Carter subgroups of $G$.
\end{cor}

\begin{proof}
Let $C$ be a Carter subgroup of $G$, and let $\lambda \in \irr C$ be linear.
Then $\Q_{o(\lambda)}$ is the field of values of some $\chi \in \irr G$,
by Theorem A.  Hence $\lambda^2=1$. Since $C$ is nilpotent, it follows that $C$ is
a $2$-group, and therefore $C$ is contained in some Sylow 2-subgroup of $G$.
Since $C$ is self-normalizing, we have that $C=P$, as desired. 
\end{proof}

Another (character theoretic) proof of Corollary \ref{corB} can be deduced using one of the
main  results in \cite{Is73}.
Indeed: by Theorem 10.9 of \cite{Is73}, we have a canonical bijection ${\rm Irr}_{2'}(G) \rightarrow {\rm Irr}_{2'}(\norm GP)$, where ${\rm Irr}_{2'}(G)$ are the odd-degree irreducible characters of $G$.
If $\norm GP>P$, then there is a linear character in $\norm GP$ of odd order $n>2$.
Hence, $\Q_n$ is the field of values of some irreducible character of $G$.

\medskip

Many finite groups (but not all, of course) have Carter subgroups and it has been proven 
that they are   conjugate whenever they exist
(see \cite{V09}). As pointed out by B. Sambale, notice that Corollary C also holds for them
by using the main result of \cite{SF19}, which implies that a rational group
has a self-normalizing Sylow 2-subgroup (and the main result of \cite{V09}). 
I haven't extensively investigated the extent to which Corollary B applies to non-solvable groups with a Carter subgroup. Groups with exactly one irreducible rational valued character have odd order
(see Theorem 6.7 of \cite{N18}). Non-solvable groups $G$ with exactly two irreducible rational valued characters have the structure
$\Oh{2'}{G/\oh{2'}G}={\rm PSL}_2(3^{2f+1})$ (Theorem 10.2 of \cite{NT08}).
If a group with such a structure possesses a Carter subgroup $C$, it can be proved that $C$
has odd order.

\medskip
Finally, we have mentioned in the Introduction that the theory of head characters proves that $|C/C'|\le k(G)$, whenever $G$ is solvable and
$C$ is a Carter subgroup of $G$. By Theorem B of \cite{N22}, we
also have that $|C/C'|=k(G)$ if and only if $G$ is abelian.
Again, I wonder if the  inequality and the equality above hold for general finite groups possessing Carter subgroups.

\end{document}